\documentclass[11pt]{amsart}
\usepackage{amsmath,amssymb,amsfonts,enumerate,graphicx,bm,fullpage}
\usepackage{amsthm}
\usepackage{mathptmx}

\usepackage[all]{xy}
\usepackage{xcolor}
\usepackage{framed}
\input xy
\xyoption{all}

\newtheorem{Theorem}{Theorem}[section] \newtheorem{Corollary}[Theorem]{Corollary} \newtheorem{Lemma}[Theorem]{Lemma} \newtheorem{Proposition}[Theorem]{Proposition} \theoremstyle{definition}
 \newtheorem{Definition}[Theorem]{Definition} \newtheorem{Example}[Theorem]{Example}  \newtheorem{Remark}[Theorem]{Remark}  \newtheorem{Question}[Theorem]{Question}    \numberwithin{equation}{section}  \newtheorem{Algorithm}[Theorem]{Algorithm}

\DeclareMathOperator{\Char}{char}
   \DeclareMathOperator{\Ann}{Ann}

 \DeclareMathOperator{\Ht}{ht}

\newcommand{\llar}{-\kern-5pt-\kern-5pt\longrightarrow}
\def\restr{{\kern-1pt\restriction\kern-1pt}}

\title{Lower Bounds for Betti Numbers of Monomial Ideals}

\author[A.Boocher]{Adam Boocher$^1$}
\address{$^{1}$  University of Utah, Salt Lake City, Utah, USA}
\email{boocher@math.utah.edu and  aboocher@gmail.com}

\author[J. Seiner]{James Seiner$^2$}
\address{$^{2}$  University of Michigan, Ann Arbor}
\email{seiner@umich.edu}

\date{\today}

\begin{document}
\maketitle

\begin{abstract}Let $I$ be a monomial ideal of height $c$ in a polynomial ring $S$ over a field $k$.  If $I$ is not generated by a regular sequence, then we show that the sum of the betti numbers of $S/I$ is at least $2^c + 2^{c-1}$ and characterize when equality holds.  Lower bounds for the individual betti numbers are given as well. \end{abstract}

\section{Introduction}
If $I$ is a homogeneous ideal in a polynomial ring $S$ over a field $k$, the betti number $\beta_i(S/I)$ denotes the rank of the $i$-th free module appearing in a minimal $S$-free resolution of $S/I$.  The main result of this paper is the following: 
\begin{Theorem}\label{MainTheorem} Let $I$ be a monomial ideal of height $c$ in a polynomial ring $S$.  
If $I$ is not a complete intersection then $\sum \beta_i(S/I) \geq 2^c + 2^{c-1}.$  
Furthermore, equality holds if and only if the betti numbers are $\{1,3,2\}$, $\{1,5,5,1\}$, or a extension thereof by tensoring with a Koszul complex. 
By this we mean that when equality holds the generating function for $\beta_i(S/I)$ is either
$$(1 + 3t + 2t^2)(1+t)^{c-2}, \mbox{ or } (1 + 5t + 5t^2 + t^3)(1+t)^{c-3}.$$
\end{Theorem}
\medskip
Suppose $I$ is an arbitrary ideal of height $c$ and let $\beta(S/I)$ denote the sum of the betti numbers of $S/I$.   
If $S/I$ is a complete intersection (CI), then the Koszul complex is a resolution and $\beta(S/I) = 2^c$.
It has been conjectured that for arbitrary ideals, $\beta(S/I) \geq 2^c$, a fact that was only settled this year by Walker (provided $\Char k \neq 2$) \cite{W}.  
This ``Total Rank Conjecture'' is a weaker version of a conjecture due to Buchsbaum-Eisenbud \cite{BE} and Horrocks \cite{H} that if $I$ has height $c$ then $\beta_i(S/I) \geq {c \choose i}.$  If $c \geq 5$ this is wide open.   For history of this problem and results in special cases, see \cite{ABO, CESundance, D, Erman, W}.

The motivation for this paper stems from work of Charalambous, Evans, and Miller \cite{C,CE, CEM} concerning larger bounds for the sum of the betti numbers when $S/I$ is not a CI.  They proved that if $S/I$ is not a CI then  $\beta(S/I) \geq 2^c + 2^{c-1}$ provided:
\begin{enumerate}
\item $I$ has finite colength and is monomial; or
\item $I$ has finite colength and $c \leq 4$.
\end{enumerate}
Our contribution is thus to remove the finite colength assumption from (1), which is non-trivial.  
Indeed, in \cite{C,CE}, in the context of multi-graded modules of finite length, the authors proved that for monomial ideals of finite colength, if $S/I$ is not a CI then one has $\beta_i(S/I) \geq { c \choose i} + { c-1 \choose  i-1}$  from which they derive the inequality for $\beta(S/I)$ by summing.
This bound on the individual betti numbers is rather strong and is false for monomial ideals not of finite colength.  
For instance, it implies that the last betti number is always at least two, which implies the interesting fact that if $I$ is monomial of finite colength and $S/I$ is Gorenstein then it is a complete intersection - a fact that is not true if $\dim(S/I) > 1$.
Indeed, the authors noted that the sequence $\{1,5,5,1\}$ violates their bound and thus is not the betti sequence of any multi-graded module of finite length.  
However, this is the betti sequence of $S/I$ when $I = (xy,yz,zv,vw,wx) \subset S = k[v,w,x,y,z]$.  
We note that while it is true that localizing can only decrease the betti numbers, in this example, $S/I$ is a CI at each associated prime, so the result in the finite colength case (which requires $I$ to not be a CI) doesn't immediately help via localization.  
This is precisely the obstruction we address in this paper.

What is surprising about Theorem \ref{MainTheorem} is that although the bounds on the individual betti numbers discovered in \cite{C,CE} for monomial ideals of finite colength do not hold for arbitrary monomial ideals, the sum of the betti numbers is still as large as these bounds predict.
Our method is outlined in Section \ref{Section 2}. 
Roughly speaking, we reduce the problem to ideals that are complete intersections on the punctured spectrum, and then find tight bounds on the betti numbers for such ideals. 
We are able to control the sum of the betti numbers in our arguments, even though the beautiful bounds discovered in \cite{C,CE} for the finite length case cannot be extended directly.  We close by summarizing what we can say about the individual betti numbers (see Section \ref{section individual}).

In one sense it seems almost coincidental that $\{1,5,5,1\}$ sums to $2^3 + 2^2$ and by our Theorem, this is essentially the only case (along with $\{1,3,2\}$) where $\beta(S/I) = 2^c + 2^{c-1}$.   We ask the following questions:

\begin{Question}
If $I$ is a homogeneous ideal of height $c$ in a polynomial ring $S$ that is not a CI, is $$\beta(S/I) \geq 2^c + 2^{c-1}?$$  
\end{Question}

We remark that this was raised in \cite{CESundance} when $I$ has finite colength.  Given the content of this paper, it would be interesting to consider whether a proof in the finite colength case would imply an answer in general.  Finally, although we expect that the betti sequences of monomial ideals are rather special, we remark that even for general homogeneous ideals, we know of no ideal $I$ where $\beta(S/I) = 2^c + 2^{c-1}$ but where the betti numbers are different than those in Theorem \ref{MainTheorem}.

\begin{Question} If $I$ is any homogeneous ideal of height $c$ in a polynomial ring and $\sum\beta_i(S/I) = 2^c + 2^{c-1}$, then are the betti numbers of $S/I$ of the form in Theorem \ref{MainTheorem}?
\end{Question}

\subsection{Notation}
Because our analysis of monomial ideals involves referring to particular variables, we shall use the convention that all lowercase letters are assumed to be variables in $S$.   Capital letters, when used to refer to elements in a ring will denote monomials.  If $M$ is a finitely generated multi-graded $S$-module then by $\beta(M)$ we mean $\sum \beta_i(M)$.  If $I$ is generated by a regular sequence we will say that $S/I$ is a complete intersection (CI) and by an abuse of notation we will also say that $I$ is a CI.  By the support of a monomial ideal, we will mean the set of variables that appear in at least one minimal monomial generator.

\section{Reduction to Nearly Complete Intersections}\label{Section 2}
In this section we show that the proof of Theorem \ref{MainTheorem} can be reduced to a special class of ideals we call \emph{nearly complete intersections}, which we define below. 
The rough idea is that localizing an ideal should only decrease the betti numbers, and if ever we can localize to something with either a larger height, or an ideal with fewer variables in its support, then we can use induction to bound the betti numbers.   We will consider only localization at monomial prime ideals, which is essentially the same as inverting variables (see Lemma \ref{Lemma localization}).
Since Theorem \ref{MainTheorem} concerns ideals that are not complete intersections, an obstruction to this procedure will be those ideals, like $I = (xy,yz,zv,vw,wx)$ that are not CI, but such that all monomial localizations are CI. 

\begin{Remark}\label{polarization}
Since the betti numbers and height of an ideal are preserved upon polarization, (see for instance \cite[Corollary 1.6.3]{HH}) in what follows we consider only squarefree monomial ideals.
\end{Remark}

\begin{Lemma}\label{Lemma localization} Suppose that $I$ is a squarefree monomial ideal in $S$ with minimal monomial generators $g(I)$.  
Let $P$ be a monomial prime ideal, that is, a subset of the variables of $S$.  
Let $J$ be the ideal generated by the $g(I)$ after setting the variables not in $P$ equal to $1$. 
Then $\beta_i(S/I) \geq \beta_i(S/J)$.  Further, $\Ht I \leq \Ht J$.
\end{Lemma}
\begin{proof}
Since in $S_P/I_P$ all variables not in $P$ are units, it follows that $S_P/I_P = S_P/J_P$.  Since localization is exact, we know that a minimal free resolution of $S$ modules remains exact upon localization at $P$.  It will be minimal precisely when all the maps have entries in $P$.  Hence,  
$$\beta_i(S/I) \geq \beta_i(S_P/I_P) = \beta_i(S_P/J_P) = \beta_i(S/J).$$ The last equality follows since $J$ involves only variables in $P$.  The result on the height follows as $I \subset J$. 
\end{proof}

This observation is enough to recover the Buchsbaum-Eisenbud-Horrocks Rank Conjecture for monomial ideals, which is well-known:
\begin{Proposition}\label{Proposition Horrocks choose}
Suppose that $I$ is a squarefree monomial ideal and that $I$ has an associated prime $P$ of height $c$. Then $\beta_i(S/I) \geq { c \choose i}$.
\end{Proposition}

\begin{proof} Since $I$ can have no embedded primes, we see that $S_P/I_P = S_P/P_P$.  Note $P$ is a prime monomial ideal and thus a CI. Hence $\beta_i (S/I) \geq \beta(S_P/I_P) = \beta(S_P/P_P) ={c \choose i}$ by Lemma \ref{Lemma localization}.
\end{proof}

\begin{Remark}\label{Remark Modules}
This idea can also be extended to prove that if $M$ is a multi-graded module whose annihilator has height $c$ then $\beta_i(M) \geq {c \choose i}$.  For the details, see \cite[Section 4]{C}.
\end{Remark}
We will frequently make use of Lemma \ref{Lemma localization} in the case that $P$ is the ideal generated by all the variables but one variable $x$.  If this is the case, we will write $I(x=1)$ to denote the ideal $J$ described in  Lemma \ref{Lemma localization}.

\begin{Definition}We say that a squarefree monomial ideal $I$ is nearly a complete intersection (NCI) if it is generated in degree at least two, is not a CI, and for each variable $x$ in the support of $I$, $I(x=1)$ is a CI.
\end{Definition}

We now outline our basic plan of attack: 

\begin{Algorithm}\label{Algorithm}
Suppose that $I$ is a squarefree monomial ideal of height $c$ that is not a CI.  We describe the following algorithm:

\begin{itemize} \item If some variable $x$ is a generator of $I$, then choose such an $x$ and return $J$, the ideal generated by the remaining minimal generators. We say that $I$ is a cone over $J$.  Notice: \begin{itemize} \item $\Ht J = c-1$; \item $\beta(S/I) = 2\beta(S/J)$; \item If $I$ is not a CI then neither is $J$. \end{itemize}
\end{itemize}
If no variable is a generator then: 
\begin{itemize} \item If there is a variable $x$ such that $I(x=1)$ is not a CI, then choose such an $x$ and return $J=I(x=1)$.  Notice $\beta(S/I) \geq \beta(S/J)$ and $\Ht J \geq \Ht I$.
\item If for each variable $x$, $I(x=1)$ is a complete intersection then return $I$, which is NCI.
\end{itemize}
\end{Algorithm}

The following theorem will be proven in Section \ref{Section NCI}. 

\begin{Theorem}\label{ThmNCI}
If $I$ is NCI of height $c$ then $\beta(S/I) \geq 2^c + 2^{c-1}$.

Equality holds in only two cases: if $c = 2$ and the betti numbers of $S/I$ are $\{1,3,2\}$, and if $c=3$ and the betti numbers of $S/I$ are $\{1,5,5,1\}.$
\end{Theorem}

Using this Theorem we are able to prove Theorem \ref{MainTheorem}.
\begin{proof}[Proof of Theorem \ref{MainTheorem}]
By Remark \ref{polarization} we may assume that $I$ is squarefree.  If $I$ is NCI, we are done.  If not, we can iterate Algorithm \ref{Algorithm} until we arrive at a NCI ideal $J$.  
In so doing, suppose we have encountered $d$ cones.  
Then we have that $\Ht J \geq c -d$ and 
$$\beta(S/I) \geq 2^d\beta(S/J).$$
By Theorem \ref{ThmNCI} we know that $\beta(S/J) \geq 2^{\Ht J} + 2^{\Ht J -1}\geq 2^{c-d} +2^{c-d-1}$.
Thus
$$\beta(S/I) \geq 2^c + 2^{c-1}.$$

Notice that equality holds only if $\Ht J = c-d$,  $\beta(S/J) = 2^{c-d} + 2^{c-d-1}$, and at each stage of the algorithm, equality of betti numbers holds.
By Theorem \ref{ThmNCI}, this happens only if $\Ht J = 2$ or $\Ht J =3$ in which case the betti numbers of $S/J$ are respectively $\{1,3,2\}$ or $\{1,5,5,1\}$.  Thus the betti numbers of $S/I$ are given by cones on these as required.
\end{proof}

\section{Two Decomposition Techniques}
Having reduced the problem to studying NCI ideals, we roughly classify them, and compute bounds for their betti numbers. 
As we show in the next section, we require two very different techniques to bound the betti numbers. 
The first technique, developed in \cite{FHV}, comes from the world of betti splittings which gives the betti numbers of $I$ in terms of the betti numbers of the three related ideals.  
This only works in certain cases but has the benefit that everything can be stated in terms of ideals, our subject of study.
The second technique, developed in \cite{B} works in general but relates the betti numbers of $S/I$ to those of $S/(I,x)$ and the module $H = (I:x)/I$ both regarded as modules over the polynomial ring $S/(x)$.  The downside of this approach is that $H$ need not be a cyclic module, and hence induction is not possible.  We summarize these two ideas in this section.   

\begin{Proposition}[Corollary 2.7 of \cite{FHV}]\label{betti splitting}
Suppose that $I$ is a squarefree monomial ideal and $I$ can be written as $I = xJ + K$ where no generator of $K$ is divisible by the variable $x$.  If $J$ has a linear resolution then 
$$\beta_i(I) = \beta_i(J)  + \beta_i(K)  + \beta_{i-1}(J\cap K), \ \ \mbox{ for all $i$}.$$
\end{Proposition}

\begin{Proposition}[Theorem 2.3 and Proposition 2.5 of \cite{B}]\label{adam pruning}
Let $I$ be a squarefree monomial ideal and let $x$ be a variable. Let $J = (I,x)$ and regard $H := (I:x)/I$ and $S/J$, as modules over the polynomial ring $R= S/(x)$.  Then 
$$\beta_i^S(S/I) = \beta^R_i(S/J) + \beta_{i-1}^R (H).$$
\end{Proposition}

\begin{Example} Consider  $I = (uv,vw,wx,xy,yz,zu) \subset S$.  It has height $3$ and betti numbers $\{1,6,9,6,2\}$. As a module over $R = S/(x)$, we have that $S/(I,x) = R/(uv,vw,yz,zu)$ with betti numbers $\{1,4,4,1\}$.   The module $H = (I:x)/I$ is minimally generated by two elements (namely $w$ and $y$) and has the following presentation over $R$:
$$\xymatrix{ R^5 \ar[rrrr]^{\begin{pmatrix}
v &  zu & 0 & 0 &  y \\ 
0 & 0   & z & uv & -w
\end{pmatrix}}  & & & & R^2 \ar[r]& H.}$$
Its betti numbers are $\{2,5,5,2\}$. 
\end{Example}

We are able to explicitly write down a presentation for $H$, which will be helpful in computing $\beta_i^R(H)$.  

\subsection{The Presentation Matrix}
Let $H = (I:x)/I$ and regard $H$ as an $R = S/(x)$ module.  Clearly, if $I$ is a squarefree monomial ideal, and $xF_1,\ldots, xF_n$ are those minimal generators divisible by $x$ then the images of the $F_i$ will generate $H$.  Hence we have a surjective map
$$\xymatrix{R^n\ar[r]^{\phi} &  H.}$$ 
We seek a set of generators for the kernel of $\phi$.   Let $e_1, \ldots, e_n$ denote the usual basis of $R^n$.  If $cF_i \in I$ then clearly $c e_i \in \ker \phi$.   It is easy to see that these are precisely the vectors of the form $g e_i$ in the kernel of $\phi$.  The set of minimal generators of $\ker \phi$ of this form is
$$\Omega = \{ c e_i  | c \mbox{ is a minimal generator of } (I:F_i)\}.$$
An element $\sum c_j e_j$ is in $\ker\phi$ if and only if $c_{j} \in R$ and $\sum c_j F_j \in I$.   Since the $F_i$ are monomials and $I$ is a monomial ideal this condition is that the non-canceling terms of this sum are in $I$.  Let $v = \sum c_j e_j\in \ker \phi$.  We subtract off multiples of elements in $\Omega$ if necessary to assume that $\sum c_j F_j = 0$.  But such $c_j$ are just syzygies of the ideal $(F_i)$ in the polynomial ring $R$.   Generators can be computed by (for instance) the Taylor complex. We have proven: 

\begin{Theorem}\label{presentation matrix}
Let $I$ be a squarefree monomial ideal in $S$ and suppose that $xF_1, \ldots, xF_n$ are the minimal generators of $I$ that are divisible by $x$.    Let $N$ be the block diagonal matrix, the $i$th block of which is the row matrix consisting of the minimal generators of $I:F_i$ (over $R$). Let $P$ be the matrix whose columns are the minimal syzygies of the ideal generated by the $(F_i)$.  Then the block matrix $M = (N | P )$ is a presentation matrix for $H$.
\end{Theorem}

\begin{Example}
Consider the following ideal $I$ of height $4$:
$$I = (xa,xb,xcd, ah,ak,bh,bk,ac,ad,bc,bd,hk), \ \  \{\beta_i(S/I)\} = \{ 1, 12, 30, 34, 21, 7, 1\}.$$
The presentation matrix for $H = (I:x)/I$ will have three rows - one for the generators $a,b,cd$ respectively.  Theorem \ref{presentation matrix} says a presentation matrix is: 
$$\left(\begin{array}{ccccccccccccccc}
c & d & h & k & 0 & 0 & 0 & 0 & 0 & 0 & 0  &  b  & hk & 0\\
0 & 0 & 0 & 0 & c & d & h & k & 0 & 0 & 0  &  -a & 0   & hk\\
0 & 0 & 0 & 0 & 0 & 0 & 0 & 0 & a & b & hk & 0  & -a  & -b
\end{array}\right)$$
Notice that the last two columns are not minimal relations.   Thus the following is actually a minimal presentation matrix, and notice it is block diagonal:
$$\left(\begin{array}{ccccccccc|ccc}
c & d & h & k & 0 & 0 & 0 & 0 & b & 0 & 0 & 0    \\
0 & 0 & 0 & 0 & c & d & h & k & -a& 0 & 0 & 0   \\
\hline
0 & 0 & 0 & 0 & 0 & 0 & 0 & 0 & 0 & a & b & hk   
\end{array}\right).$$  This means that $H$ has $R/(hk,a,b)$ as a direct summand and exemplifies the following Corollary.  The betti numbers of $H$ are $\{3,12,19,15, 6,1\}$ and the betti numbers of $S/(I,x)$ (as an $R$-module) are $\{1,9,18,15,6,1\}$. 
\end{Example}

\begin{Corollary}\label{presentation splits}
Let $x\in S$ be a variable and let $R = S/(x)$ be the polynomial ring in one fewer variable.  Suppose that $n\geq 2$ and $F_1, \ldots, F_n$ are squarefree monomials none divisible by $x$ that form a regular sequence. 
Let $K$ be a squarefree monomial ideal none of whose generators are divisible by $x$.  Let 
$$I = x(F_1, \ldots, F_n) + K.$$
Suppose that $F_iF_n \in I$ for $i = 1, \ldots, n-1$.  
Let $L = I:F_n$.  Then as $R$-modules, 
$$H := \frac{I:x}{I} \cong \frac{R}{L} \oplus H'$$
where $H'$ is a nonzero module.  Then 
$$\Ht \Ann H' \geq \Ht \Ann H \geq \Ht I -1, \ \ \Ht L \geq \Ht I-1.$$
and $\beta^R(H) = \beta^{R}(S/L) + \beta^{R}(H')$.
\end{Corollary}

\begin{proof}
Consider the presentation matrix $M$ in Theorem \ref{presentation matrix}.  $P$ will be the first syzygy matrix on the $F_i$, which we can take to be the first matrix in the Koszul complex on the $F_i$.  Those columns of $P$ whose last entry is nonzero are of the form $F_n e_i - F_i e_n$ for $i = 1, \ldots, n-1$ where $g_i = \gcd(F_i, F_n)$.  Since $F_n F_i \in I$, both terms of this sum are syzygies themselves and appear as columns of $N$, so these syzygies in $P$ are non-minimal and are not necessary.  We may assume the last row of $P$ is zero.  Since $N$ is a block diagonal matrix, this allows us to write $M$ as a block diagonal matrix, $M = (L | M')$ where $L$ is the bottom row of $N$ and $M'$ is the rest. Finally, notice that $\Ann H \subset \Ann H'$, so the result on height follows.
\end{proof}

\begin{Corollary}\label{presentation split betti number} Suppose that $I$ is a squarefree monomial ideal of height $c$ satisfying the condition in the previous Corollary.  Then for all $i$,
\begin{eqnarray*} 
\beta_i(S/I) &\geq &{ c \choose i} + { c-1 \choose i-1}.\end{eqnarray*}
Then $\beta(S/I) \geq  2^c + 2^{c-1}$ and equality holds only if $S/(I,x)$ is a complete intersection.
\end{Corollary}
\begin{proof}
By Proposition \ref{adam pruning}, Corollary \ref{presentation split betti number}, Proposition \ref{Proposition Horrocks choose}, and Remark \ref{Remark Modules} we have that
\begin{eqnarray*} 
\beta_i(S/I) &= &\beta^{R}_i(S/(I,x)) + \beta^{R}_{i-1}(S/L) + \beta^R_{i-1}(H')  \\
		  & \geq & {c-1 \choose i}  + { c-1 \choose {i-1}}  + { c-1 \choose {i-1}} \\
		  & =  & { c \choose i} + { c-1 \choose i-1}. \\
\beta(S/I)    &  \geq & 2^c + 2^{c-1}. \end{eqnarray*}
We obtain the inequality because the modules appearing on the right are of height at least $c-1$.   The assertion on $\beta(S/I)$ follows by taking sums.  If $S/(I,x)$ is not a CI then the inequality will be strict as then $\beta_1^R(S/(I,x)) > c-1$.  
\end{proof}

\begin{Remark}\label{prune is sharp} The ideal $I = (xy,xz,yz,u_1, u_2, \ldots, u_{c-2})$ illustrates that these inequalities are sharp.
\end{Remark}


\section{Properties of NCI ideals}\label{Section NCI}

\begin{Lemma}\label{Lemma NCI}Suppose that $I$ is NCI.  If $m_1$ and $m_2$ are two minimal monomial generators of $I$ then their gcd has degree at most $1$.
\end{Lemma}
\begin{proof}
If $x$ and $y$ are distinct variables that divide $m_1$ and $m_2$ then $I(x=1)$ is not a CI since $m_1/x$ and $m_2/x$ are minimal generators with a common factor.
\end{proof}

\begin{Lemma}\label{everything has a gcd}  Suppose that $I$ is NCI and $F$ is a minimal generator of $I$.  Then $F$ must have a factor in common with some other generator.
\end{Lemma}
\begin{proof}
Since $I$ is not a CI there are two minimal generators $M_1,M_2$ that have a factor in common. Since $F$ is a monomial of degree at least two, let $x$ and $y$ be two variables that divide $F$, and assume that $x,y$ do not appear in any other minimal generator.  Then the generators of $I(x=1)$ are the same as those of $I$ except that $F$ is replaced with $F/x$.  $M_1$ and $M_2$ will still be minimal generators, since they are not divisible by $(F/x)$ which has $y$ as a factor.  
\end{proof}

\noindent \textbf{Notation $(\star)$:}  For the remainder of this section we will assume that $I$ is NCI of height $c$ and assume that each associated prime of $I$ has height $c$.  If $x$ is in the support of $I$, and $xF_1, \ldots, xF_n$ are those generators of $I$ divisible by $x$ ($n \geq 2$) then we may write 
$$I = x(F_1, \ldots, F_n) + J  + K$$ 
where $J$ consists of those remaining generators in the ideal generated by the $F_i$ and $K$ is the ideal generated by the remaining generators (if any).  
Such a decomposition exists for any variable $x$.  Notice that $I(x=1) = (F_1, \ldots, F_n) + K$, which must be a CI.  It must necessarily be of height $c$, since there is a minimal prime of $I$ that does not contain $x$ (since $x\notin I$).  This implies that $K$ has to be a complete intersection of height $c-n$ and that the variables appearing in $K$ are distinct from those appearing in the $F_i$.  We will use this notation throughout this section.

\begin{Lemma}
In the above notation, at most one of $F_1, \ldots, F_n$ has degree greater than 1.   If $F_n$ has degree greater than $1$ then $J\subset (F_1, \ldots, F_{n-1})$.
\end{Lemma}
\begin{proof}
Notice that 
\begin{itemize}
\item $\Ht (J+ K) \geq c-1$ since if $P$ is a minimal prime of $J+K$, then $(P,x)$ will be a prime containing $I$.
\item $\Ht K = c-n$ as discussed above\end{itemize}
We conclude the height of $J$ is at least $n-1$.  Therefore, $J$ is not contained in an ideal generated by only $n-2$ of the $F_i$.  Thus without loss of generality, $J$ contains minimal generators $F_iG_i$ for $i = 1, \ldots, n-1$.   Each generator has $\gcd(F_iG_i,xF_i) = F_i$, and thus by Lemma \ref{Lemma NCI} $F_i$ has degree one.  The final claim follows since if $J$ had a minimal generator $G$ divisible by $F_n$ then $\gcd(G, xF_n)$ would have degree greater than $1$, a contradiction. 
\end{proof}

For $i = 1, \ldots, n-1$ we now denote $F_i$ by $a_i$ to stress it is a variable.  Thus we refine Notation $(\star)$
\begin{equation}\label{The format} I = x(a_1, \ldots, a_{n-1}, F_n) + J + (h_1K_1, \ldots, h_{c-n}K_{c-n})\end{equation}
where $(a_1, \ldots, a_{n-1}, F_n, h_iK_i)$ is a regular sequence and $J\subset (a_1, \ldots, a_{n-1},F_n)$.   

\begin{Proposition}\label{The NonlinearCase}
Suppose that $I$ is an NCI of height $c$ with a minimal generator of degree at least three.  Then $\beta(S/I) > 2^c + 2^{c-1}$.  In addition $\beta_i(S/I) \geq  {c \choose i} + { c-1 \choose i-1}$ for all $i$. 
\end{Proposition}
\begin{proof} First, if $I$ has an associated prime of height greater that $c$ then by Proposition \ref{Proposition Horrocks choose}, we would have that $\beta(S/I) \geq 2^{c+1}$, which is larger than $2^c + 2^{c-1}$.   Hence we will assume that $I$ is of the form in Notation $(\star)$.

By Lemma \ref{everything has a gcd}, the generator of degree at least three will have a variable $x$ in common with at least one other generator.  Hence we may assume 
$$I = x(a_1, \ldots, a_{n-1}, F_n) + J + (h_1K_1, \ldots, h_{c-n}K_{c-n})$$
as above, where $n \geq 2$ and $\deg F_n \geq 2$.  We will show that $a_iF_n \in I$ for $i = 1, \ldots, n-1$ and then the result will follow from Corollary \ref{presentation split betti number}.

Since $\deg F_n \geq 2$, there are two distinct variables $y,z$ so that $F_n = yzF_0$.  Note that $y,z \neq a_i$ since that would imply $xyzF_0$ is not a minimal generator.  Let us examine $I(y=1)$.  
$$I(y=1) = (xa_1, \ldots, xa_{n-1}, xzF_0) + J(y=1) + (h_1K_1, \ldots, h_{c-n}K_{c-n})(y=1).$$
This must be a complete intersection, so at most one of $xa_1, \ldots, xa_{n-1},xzF_0$ can be a minimal generator of $I(y=1)$.   But $xzF_0$ must be a minimal generator, so we have that the $xa_i$ are not minimal generators of $I(y=1)$ and this means $ya_i\in I$ for all $i$.  Thus $a_iF_n \in I$ as required.  The same argument shows that $za_i \in I$ for all $i$ as well.  We are thus able to apply Corollary \ref{presentation split betti number} and the results follow. Note that $(I,x)$ is not a CI as $ya_1, za_1 \in I$.
\end{proof}

All that remains is the case that $I$ is generated in degree two:
$$I = x(a_1, \ldots, a_n) + J + (h_1k_1, \ldots, h_{c-n}k_{c-n}).$$
Our proof proceeds in cases:

\begin{Proposition}\label{Height c Case} Suppose $I$ is a squarefree monomial ideal of height $c\geq 2$ of the form 
$$I = x(a_1, \ldots, a_c) + J$$
where $J\subset (a_1, \ldots, a_c)$.
Then 
$$\beta(S/I) \geq 2^{c+1} - 2 \geq 2^c + 2^{c-1}.$$  The second inequality is strict when $c \geq 3$. More specifically we have: 
$$\beta_1(S/I) \geq 2c-1, \ \ \beta_i(S/I) \geq 2{c\choose i} \  \mbox{ for $i\geq 2$}.$$

\end{Proposition}
\begin{proof}
Notice that $\Ht J \geq c-1$ and that $x(a_1, \ldots, a_c)\cap J = xJ$. We have
$$\beta_1(S/I) = c + \beta_1(S/J) \geq c + (c-1) = 2c-1,$$
and for $i \geq 2$, by Theorem \ref{betti splitting}
\begin{eqnarray*}
\beta_i(S/I) & =  & \beta_{i}(S/x(a_1, \ldots, a_c)) + \beta_{i}(S/J) + \beta_{i-1}(S/(x(a_1, \ldots, a_c)\cap J))\\
 & = & {c \choose {i}} + \beta_{i}(S/J) + \beta_{i-1}(S/xJ) \\
 & = & {c \choose {i}} + \beta_{i}(S/J) + \beta_{i-1}(S/J) \\ 
 & \geq & {c \choose {i}} + {{c-1} \choose i} + {{c-1} \choose i-1}  = 2{ c \choose i}.  \end{eqnarray*}
 Summing, we see that 
 $$\beta(S/I) = 1 + \beta_1(S/I) + \sum_{i \geq 2} \beta_i(S/I) \geq 1 + (2c -1) + 2( 2^c - 1 - c) = 2^{c+1} -2.\qedhere$$
\end{proof}

\begin{Example} \label{ex is sharp}The inequalities above are sharp.  Let $I = (xa,xb,xc, ad,be)$.  Then $\Ht I = 3$ and $\{\beta_i(S/I)\} = \{1,5,6,2\}$, so $\beta(S/I) =14 = 2^4 -2.$   

More generally, the family of ideals 
$$x(a_1, \ldots, a_c) + (a_2b_2, \ldots, a_cb_c)$$
has sum of betti numbers equal to $2^{c+1} -2$ as can be checked using the decomposition above.  Evidently the bounds for the individual betti numbers must be equalities as well. 
\end{Example}

\begin{Remark}We remark that if equality holds when $c = 2$ then it is clear that the betti numbers $\beta_i(S/I)$ are $\{1,3,2\}$.
\end{Remark}

The last remaining case we have is:

\begin{Proposition}\label{simple case linear resolution}
Suppose that $I$ is an NCI of the form:
$$I = x(a_1, \ldots, a_n) + J + (h_1k_1, \ldots, h_{c-n}k_{c-n})$$
where $2 \leq n < c$ and $J$ is generated in degree $2$.   Then $\beta(S/I) \geq 2^{c} + 2^{c-1}.$  The inequality is strict unless $n = 2$ and $c=3$. 

If $n = c-1$ then 
$$ \beta_i(S/I) \geq {c \choose i} + {c-1 \choose i} \ \mbox{ if  $1 \leq i \leq c.$}$$

If $n < c-1$ then 
$$ \beta_i(S/I) \geq {c \choose i} + {c-1 \choose i-1} \ \mbox{  if $0 \leq i \leq c.$}$$
\end{Proposition}

\begin{proof} First, if $I$ has an associated prime of height greater that $c$ then by Proposition \ref{Proposition Horrocks choose}, we would have that $\beta(S/I) \geq 2^{c+1}$, which is larger than $2^c + 2^{c-1}$.   Hence we will assume that $I$ is of the form in Notation $(\star)$.

Let $K = (h_1k_1, \ldots, h_{c-n}k_{c-n})$.   Notice that $\Ht (J+K) \geq c-1$.  Then by Theorem \ref{betti splitting} we have that 
\begin{eqnarray*}
\beta_1(S/I) & = & \beta_1(S/x(a_1, \ldots, a_n)) + \beta_1(S/(J+K)) \\
 & \geq &  n + c - 1. \\ 
\beta_i(S/I) & = & \beta_i(S/x(a_1, \ldots, a_n)) + \beta_i(S/(J+K)) + \beta_{i-1}(S/(x(a_1,\ldots,a_n)\cap (J+K))) \\ 
&= & {n \choose i} + \beta_i(S/(J+K)) + \beta_{i-1}(S/(a_1, \ldots, a_n)\cap (J+ K)) \\
&\geq & {n \choose i} + {c-1 \choose i} + {m \choose i-1} \mbox{ for $i \geq 2$}.
\end{eqnarray*}
Where $m= \min(n, c-1) \leq \Ht ((a_1, \ldots, a_n) \cap (J+K))$.  Then we have that 
$$\beta(S/I) \geq 2^n + 2^{c-1} + 2^m - 2.$$
\textbf{Case 1: $n = c-1$:} The inequalities simplify to
$$ \beta_1(S/I) \geq 2c-2$$
$$\beta_i(S/I) \geq {c \choose i} + {c-1 \choose i} \ \ i \geq 2.$$
which yield $\beta(S/I) \geq 2^c + 2^{c-1} -2$.

However notice that equality occurs only if $(J+K)$ is a CI.  We will rule this out.  Indeed, consider $I(h_1 =1)$.   This has $xa_1,xa_2$ as minimal generators and thus, either $a_1h_1$ or $a_2h_1$ is in $I$. But then $(J+K)$ contains $h_1k_1$ also, so that $(J+K)$ is not a CI and $\beta_1(S/I) \geq 2c-1$.  Now $\beta(S/I)$, which is even, must be at least $2^c + 2^{c-1} -1$.  The result follows.

If $c\geq 4$, we see from examining $I(h_1=1)$, that there are $(c-2)$ generators of the form $a_ih_1$ and $(c-2)$ of the form $a_ik_1$.   Since $h_1k_1$ is also a minimal generator, 
$$\beta_1(S/(J+K)) \geq (c-2) + (c-2)  + 1 = 2c-3 \geq 2 + (c-1).$$   Thus one bound from Proposition \ref{Proposition Horrocks choose} is off by at least 2, so  $\beta(J+K) \geq (2^{c-1} -1) + 2$.   Similarly, $(a_1, \ldots, a_n)\cap (J+K)$ includes all the same generators $a_ik_1, a_ih_1$, so 
$$\beta_1(S/J)  \geq (c-2) + (c-2) = 2c - 4\geq 1 + (c-1).$$ Thus we have that $\beta(S/I) \geq 2^c + 2^{c-1} + 1$ as required. 

\noindent \textbf{Case 2: $n \leq c-2$:}  We will assume that any variable $y$ divides at most $n$ minimal generators.  In other words, we have chosen the $x$ that divides the largest number of generators.
Consider the ideal $I(h_1= 1)$, which must be a CI.   This ideal contains $xa_1, \ldots xa_n$.  At most one of these can be a minimal generator.  Thus there must be minimal generators in $J(h_1=1)$ that divide $n-1$ of these terms.  Without loss of generality, say $h_1a_1, \ldots, h_1a_{n-1} \in J$. These are minimal generators.  Thus $h_1$ divides $n$ generators and by assumption, it divides no other generators.  In particular, $h_1a_n \notin I$.

Now since $n \leq c-2$, $h_2k_2 \in I$.  Observe that $I(h_2 = 1)$ contains $xa_1, \ldots, xa_{n-1}$ and $h_1a_1, \ldots, h_1a_{n-1}$.  As $xh_2 \notin I$ and $h_1h_2\notin I$ we must have that $h_2a_1, \ldots, h_2a_{n-1} \in I$ and as before, $h_2a_n \notin I$.

Finally, consider $I(a_n = 1)$.  This ideal contains $h_1a_i$ and $h_2a_i$ for $1 \leq i \leq n-1$.   This implies that $a_na_i \in I$ for each $I$.   We are now in the case of Corollary \ref{presentation split betti number}.  Notice that $(I,x)$ is not a CI since it contains $h_1a_1$ and $h_2a_1$ so the inequality is strict.
\end{proof}

\begin{Remark}
If $I$ is NCI of height $3$ and $\beta(S/I) = 2^3 + 2^2$ then from the proofs above, $I$ has exactly $5$ quadratic generators and up to relabeling, $I$ must be of the form 
$$I = x(a_1,a_2) + (a_1h_1, a_2k_1) + (h_1k_1)$$ which is the second case in Theorem \ref{MainTheorem}. 
\end{Remark}

\begin{proof}[Proof of Theorem \ref{ThmNCI}]
If $I$ has a minimal generator of degree at least three, the result follows from Proposition \ref{The NonlinearCase}.  If $I$ is generated in degree two, then the result follows from Propositions \ref{Height c Case} and \ref{simple case linear resolution}, which include the cases where equality holds.
\end{proof}


\section{The individual betti numbers}\label{section individual}
In \cite{C} it was shown that if $M$ is a multi-graded module of finite length over $S = k[x_1, \ldots, x_c]$ and $M$ is not isomorphic to $S$ modulo a regular sequence then either 
$\beta_i(M) \geq {c \choose i} + {c-1 \choose i-1}$ for all $i$  or $\beta_i(M) \geq {c \choose i} + {c-1 \choose i}$ for all $i.$  This means that, for instance, the first or last betti number must be at least $2$.  Such bounds will not hold without the finite length condition, even in the multi-graded case.

The results in this paper can be assembled to give general bounds for the numbers $\beta_i(S/I)$ when $I$ is a monomial ideal.  Suppose that $I$ is a squarefree monomial ideal of height $c$.  Then by Algorithm \ref{Algorithm} we have that
$$\beta_i(S/I) \geq \beta_i(S/(J,u_1,\ldots, u_{c-d})$$
where $J$ is NCI of height $d$.  Then  by Propositions \ref{The NonlinearCase}, \ref{Height c Case}, and \ref{simple case linear resolution} for $i \geq 1$ we have that
\begin{eqnarray}
\label{eqn1}\beta_i(S/J) &\geq & 2{d \choose i} \ \ \mbox{ for all $i\geq 2$ and $\beta_1(S/J) \geq 2d-1$ or } \\ 
\label{eqn2}\beta_i(S/J) &\geq & {d \choose i} + {d-1 \choose i-1} \ \ \mbox{ for all $i\geq 0$  or } \\
\label{eqn3}\beta_i(S/J) &\geq & {d \choose i} + { d-1 \choose i} \ \ \mbox{ for all $i\geq 1$ }
\end{eqnarray}

Then notice that the betti numbers of $S/I$ can be obtained from those of $S/J$ by tensoring with the appropriate Koszul complex on the $u_i$.   In terms of generating series: 
\begin{equation}\label{eqn4}\sum\beta_i(S/I)t^i = \left(\sum \beta_i(S/J)t^j\right)(1 + t)^{c-d}.\end{equation}
Unfortunately, because in (\ref{eqn1}) the and (\ref{eqn3}), the formula is different for $i = 0,1$, and $i = 0$ respectively, it doesn't follow that similar bounds exist for $S/I$, say with $d$ replaced by $c$, as seen in the following Example. 

\begin{Example}
Given that $(\ref{eqn1})$ is considerably larger than the other two bounds, it is reasonable to ask that if $I$ is an ideal of height $d$ whether or not at least one of (\ref{eqn2}) or (\ref{eqn3}) holds.  If $d = 4$, then this would say that the betti sequence of $S/I$ is at least as big as $\{1, 5, 9, 7, 2\}$ or $\{1, 7, 9, 5, 1\}$.  However, if $I = (xy,yz,zv,vw,wx,u)$ then the betti numbers are $\{1,6,10,6,1\}$ which violate both bounds.    Hence the bounds determined by (\ref{eqn4}) are perhaps the best we can hope for.

\medskip 
\noindent Example \ref{ex is sharp} and Remark \ref{prune is sharp} show that (\ref{eqn1}) and (\ref{eqn2}) are sharp.  

\medskip
\noindent Finally, notice that equality in (\ref{eqn3}) is impossible, as the sum of the numbers (with $\beta_0 = 1$) on the right hand side is $2^c + 2^{c-1} -1$.  Thus at least one of the betti numbers is at least one larger.  If $I = (xy,yz,zv,vw,wx)$ then the betti numbers are $\{1,5,5,1\}$ which are as close to the bound $\{1,5,4,1\}$ as possible.
\end{Example}

\section*{Acknowledgments}
\noindent The first author was partly supported by the NSF RTG grant DMS \#1246989 and the second author was supported by the University of Utah.  The authors thank the University of Utah and its summer REU program which was the start of this project.  The authors thank Srikanth Iyengar, Jake Levinson, and Jonathan Monta\~no for helpful conversations and suggestions on an earlier draft of this paper. 
\bibliographystyle{plain}
\bibliography{References}
\end{document}